\documentclass{article}
\usepackage[hyphens]{url}
\usepackage[]{hyperref}
\usepackage[]{amssymb,amsthm,amsmath}
\usepackage[]{geometry}
\usepackage[style=alphabetic,sortcites]{biblatex}

\hypersetup{
    breaklinks=true,
    colorlinks=true,
    pdfusetitle=true
}

\DeclareLabelalphaTemplate{
  \labelelement{
    \field[final]{shorthand}
    \field{label}
    \field[strwidth=1,strside=left,ifnames=1]{labelname}
    \field[strwidth=1,strside=left]{labelname}
  }
}

\DeclareFieldFormat{extraalpha}{#1}

\newtheorem{theorem}{Theorem}
\newtheorem{conjecture}[theorem]{Conjecture}
\theoremstyle{definition}
\newtheorem{definition}{Definition}

\title{The (Symbolic and  Numeric) Computational Challenges of Counting 0-1 balanced matrices}
\author{Robert Dougherty-Bliss \and Christoph Koutschan \and Natalya Ter-Saakov \and Doron Zeilberger}
\date{\today}

\begin{document}

\maketitle

\centerline{\it Dedicated to our hero, Neil J.~A. Sloane (b.~Oct. 10, 1939), on
his 85th birthday.}

\begin{abstract}
    \noindent A chessboard has the property that every row and every column has
    as many white squares as black squares. In this mostly methodological note,
    we address the problem of counting such rectangular arrays with a fixed
    (numeric) number of rows, but an arbitrary (symbolic) number of columns. We
    first address the ``vanilla" problem where there are no restrictions, and
    then go on to discuss the still-more-challenging problem of counting such
    binary arrays that are not permitted to contain a specified (finite) set of
    horizontal patterns, and a specified set of vertical patterns.  While we
    can rigorously prove that each such sequence satisfies some linear
    recurrence equation with polynomial coefficients, actually finding these
    recurrences poses major  {\it symbolic}-computational challenges, that we
    can only meet in some small cases. In fact, just generating as many as
    possible terms of these sequences is a big {\it numeric}-computational
    challenge. This was tackled by computer whiz Ron H. Hardin, who contributed
    several such sequences, and computed quite a few terms of each. We extend
    Hardin's sequences quite considerably. We also talk about the much easier
    problem of counting such restricted arrays without balance conditions.
\end{abstract}

\section*{Preface: How it all started} 

Like many combinatorial problems \cite{Condorcet,Derange}, ours was inspired by
real life. A few weeks ago, the New York Times magazine started publishing a new
kind of logic puzzle that they call {\it Not Alone}, created by Presanna
Seshadri. You are given a $6 \times 6$ (or $8 \times 8$) array of boxes with
most of them empty, but a few of them are filled with either a solid circle,
that we will denote by $1$, or an empty circle, that we will denote by $0$. The
solver has to, presumably using logic and human cleverness, fill-in the empty
boxes such that the following conditions are met:

\begin{itemize}
    \item Every row and every column must have as many zeroes as ones (i.e.,
        they each must contain $3$ zeroes and $3$ ones in the $6 \times 6$
        case).

    \item It is forbidden that on any row, and on any column, a {\bf single}
        zero will be `all alone' between two ones and
        that a {\bf single} one will be all alone between two zeroes. In other
        words the patterns $010$ and $101$ are forbidden both horizontally and
        vertically.
\end{itemize}

\paragraph{Three Natural Enumeration Problems}
Being {\it enumerators}, the following three questions immediately came to mind.

\begin{itemize}
    \item For a {\it fixed}, `numeric', positive integer $k$, but an {\it
        arbitrary}, `symbolic' $n$, how many $2k \times 2n$ 0-1  balanced
        matrices are there? In other words how many $2k \times 2n$ 0-1 
        matrices are there where every row has $n$ zeroes and $n$ ones and
        every column has $k$ zeroes and $k$ ones?

        Looking up some numbers from this problem leads to a family of
        sequences submitted by Ron H.~Hardin \cite{H}, who has made a number of
        interesting submissions to the OEIS. Some of these were detected by an
        automated search of the OEIS for recurrences by Kauers and Koutschan
        \cite{search}, and later proven to satisfy those recurrences by
        Dougherty-Bliss and Kauers \cite{hardinian}. We will tell a similar
        story here.

        We will show that the family of sequences $b_k(n)$, the number of $2 k
        \times 2 n$ balanced matrices, is ``D-finite'' for every fixed~$k$.
        That is, it satisfies a linear recurrence relation with polynomial
        coefficients. The approach will be to use {\it `holonomic nonsense'}
        \cite{Z,AZ,dfinite} which guarantees the existence of such a recurrence
        and also outlines a method to construct it. Using the very efficient
        implementations of \cite{K,Kfast,Kthesis} we---or rather our beloved
        computers---were able to compute these recurrences for $k = 2$ and $k =
        3$.
        
        Alas, it seems too difficult at present to determine rigorously
        what recurrence $b_4(n)$, $b_5(n)$, and so on might satisfy in~$n$.
        The deterministic algorithms just take too long to run. At least
        for $b_4(n)$, we succeeded to generate enough data to make a
        conjecture, using a recent and novel approach for guessing
        recurrences from little data~\cite{KK22}. For $b_k(n)$ with $k\geq5$
        we unfortunately were not able to compute enough data to guess
        recurrences.

    \item For an arbitrary (finite) alphabet (not just $\{0,1\}$) and arbitrary
        finite sets of forbidden
        horizontal and vertical patterns, $H$ and $V$, how many $k \times n$
        matrices are there avoiding the patterns of $H$ in the rows and the
        patterns of $V$ in the columns (with no balance conditions)? We will
        show that these sequences are much easier, both conceptually and
        computationally, since they always satisfy linear recurrence equations
        with {\bf constant} coefficients, or equivalently, their generating
        function is a {\it rational function}.

    \item Going back to the alphabet $\{0,1\}$, for a specific $k$, how many
        $2k \times 2n$ balanced  0-1
        matrices are there that also avoid a prescribed horizontal set of
        patterns $H$ and (another or the same) prescribed  vertical set of
        patterns $V$. Once again, we will show that for each such scenario, the
        enumerating sequence satisfies {\it some}  linear recurrence equation 
        with polynomial coefficients (in~$n$). Alas finding it is yet harder
        than the `vanilla' case above. Once again this leads to numeric
        challenges. In particular, it turns out that for the original New York
        Times puzzles where $k=n=3$ and $H=V=\{010,101\}$ that  number is {\bf
        exactly} $368$. It is relatively easy to actually construct the set of
        all such legal matrices, {\it once and for all}. It follows that,
        surprisingly, a  pure {\it brute-force} algorithm for solving these
        original puzzles is more efficient than using logic, as a human solver
        would. Just try out all $368$ possible answers and see which one agrees
        with the given clues. For the $8 \times 8$ puzzles that started to
        appear shortly after, we found that there are $34586$ possible answers,
        which suggests that it may be more efficient to do it the human way of
        using logic.
\end{itemize}

The Maple package {\tt NotAlone.txt}, available from

\url{https://sites.math.rutgers.edu/~zeilberg/tokhniot/NotAlone.txt}

solves and creates such puzzles. Procedure {\tt Ptor} implements the
brute-force approach that is optimal for the $6 \times 6$ case. Procedure {\tt
SolveN} does it in a way a human would tackle it.

\section*{Theorems}

In this section, that is purely {\it theoretical}, we will prove that the first
and third kind of sequences above are P-recursive (aka {\it holonomic}), in
other words are guaranteed to satisfy {\it some} linear recurrence equation
with {\bf polynomial coefficients} (see \cite{KaP} chapter 7), while any sequence
that comes from the second kind of enumeration problems belongs to the simpler
class of C-finite sequences (\cite{KaP}, chapter 4), i.e., satisfies {\it some}
linear equation with {\bf constant coefficients}.

\begin{theorem}
    Let $k$ be a specific positive integer, and let $n$ be a general positive
    integer. Let  $b_k(n)$ be the number of balanced $2k \times 2n$ 0-1
    matrices,
    i.e., binary matrices with $2k$ rows and $2n$ columns where every row has
    exactly $n$ ones (and hence exactly $n$ zeroes), and every column has
    exactly $k$ ones (and hence exactly $k$ zeroes). Then the sequence
    $\{b_k(n)\}_{n=1}^{\infty}$ is holonomic. In other words there exists a
    positive integer $L$ (the order) and polynomials in $n$, $p_i(n), 0 \leq i
    \leq L$, with $p_L(n)\neq0$ such that
    $$
    \sum_{i=0}^{L} p_i(n) b_k(n+i) =0 \quad .
    $$
\end{theorem}

\begin{proof}
Let $e_k(x_1, \dots, x_n)$ be the {\bf elementary symmetric function} of degree $k$:
$$
e_k(x_1, \dots, x_n) \, = \, \sum_{1 \leq i_1 <i_2 <\dots <i_k \leq n} x_{i_1} \cdots x_{i_k} \quad .
$$
It is readily seen that $b_k(n)$ is the coefficient of $x_1^n \cdots x_{2k}^n$
in $e_k(x_1, \dots,x_{2k})^{2n}$. Indeed, each monomial of $e_k(x_1,
\dots,x_{2k})$ corresponds to a way of placing $k$ ones (and $k$ zeroes) in any
particular column, making each column balanced. $e_k(x_1, \dots,x_{2k})^{2n}$
then is the weight enumerator of all column-balanced $2k$ by $2n$ 0-1
matrices. The coefficient of $x_1^n \cdots x_{2k}^n$ collects those that are
also row-balanced.
Hence
$$
b_k(n)= {\rm Coeff}_{x_1^0 \cdots x_{2k}^0} \left ( \frac{e_k(x_1, \dots, x_{2k})^2}{x_1 \cdots x_{2k}} \right )^{\!n} 
\,=\,
\left(\frac{1}{2\pi i}\right)^{\!2k} \int \left ( \frac{e_k(x_1, \dots, x_{2k})^2}{x_1 \cdots x_{2k}} \right )^{\!n} \frac{dx_1 \cdots dx_{2k}}{x_1 \cdots x_{2k}} \quad,
$$
where the integration is over the multi-circle $|x_1|=1, \dots, |x_{2k}|=1$. 
Since the integrand is holonomic in the $2k$ continuous variables $x_1, \dots,
x_{2k}$ and the one discrete variable $n$,  it follows from {\it algorithmic
proof theory} \cite{Z,AZ,K} that integrating away the $2k$ continuous variables
leaves $b_k(n)$ holonomic in the surviving discrete variable $n$.
\end{proof}

So far our alphabet was $\{0,1\}$. In the next theorem (answering the second
question above) we will be more general, but we need to introduce some
definitions.

\begin{definition}
    Fix a finite alphabet $A$ once and for all.  A word $w_1 \dots w_n$ in the
    alphabet $A$ {\bf contains} the word $p_1 \dots p_k$ (called a {\it
    `pattern'}) if there is a location $i$ such that  $w_i=p_1, \dots,
    w_{i+k-1}=p_k$. For example, with the Latin alphabet, {\it robert} contains
    the words {\it rob}, {\it obe}, {\it t}, and many others. A word $w$ {\it
    avoids} the pattern $p$ if it does not contain it. For example $101010001$
    avoids $11$.
\end{definition}

\begin{theorem}
    Let $A$ be an arbitrary (finite) alphabet, and $H$ and $V$ be arbitrary
    finite sets of words in $A$. Let $k$ be a fixed (numeric) positive integer.
    Let $m_k(n)=m_{A,H,V,k}(n)$ be the number of $n \times k$ matrices with
    entries in $A$ such that every row avoids the patterns in $H$, and every
    column avoids the patterns in $V$, then
    the sequence $\{m_k(n)\}_{n=1}^{\infty}$ satisfies a linear recurrence
    equation
    with {\bf constant} coefficients.
    In other words there exists a positive integer $L$ and numbers $c_0,c_1,
    \dots, c_L$ such that
    $$
    \sum_{i=0}^{L} c_i \, m_k(n+i) \, = \, 0  \quad .
    $$
    Equivalently, there exist polynomials $P(t)$ and $Q(t)$ (where $Q(t)$ has
    degree $L$) such that
    $$
    \sum_{n=0}^{\infty} m_k(n)\,t^n \, = \, \frac{P(t)}{Q(t)} \quad.
    $$
\end{theorem}

\begin{proof}
    Let $B$ be the set of words of length~$k$ in the alphabet~$A$ that avoid
    the patterns in $V$. This is a finite set. We will view the $k \times n$
    matrix as a one-dimensional word in this meta-alphabet. Then the
    restrictions that the rows avoid the patterns in $H$ translate to many
    conditions about pattern avoiding in this meta-alphabet. This gives rise to
    a so-called type-3 grammar, or finite automaton, whose enumerating
    generating functions are famously rational functions. In order to actually
    find them one can use the {\it positive} approach, using the
    transfer-matrix method (\cite{S}, ch. 4), or the {\it negative} approach, using
    the powerful Goulden-Jackson method, nicely exposited in \cite{NZ}.
\end{proof}

\paragraph{Comment}
For the motivating example (the {\it Not Alone} puzzles), $A=\{0,1\}$ and
$H=V=\{010,101\}$.

The next theorem states that if one counts $2k \times 2n$ balanced 0-1
matrices and imposes {\it arbitrary} horizontal and vertical conditions, the
resulting sequences are still holonomic.

\begin{theorem}
    Let $k$ be a specific positive integer, and let $n$ be a general positive
    integer. Let $H$ and $V$ be finite sets of words (`patterns') in $\{0,1\}$.
    Let  $b_{H,V,k}(n)=b_k(n)$ be the number of balanced $2k \times 2n$ 0-1
    matrices, that avoid the patterns of $H$ in every row and the patterns of
    $V$ in every column, then there exists a positive integer $L$ and
    polynomials $p_i(n)$, $0 \leq i \leq L$, with
    $p_L(n)\neq0$ such that
    $$
    \sum_{i=0}^{L} p_i(n) b_k(n+i) =0 \quad .
    $$
\end{theorem}

\begin{proof}
    Instead of {\it naive counting} where the weight of a $2k \times
    2n$ matrix was simply $t^{2n}$, we now introduce $2k$ formal variables $x_1,
    \dots , x_{2k}$ and assign a {\it weight} of a matrix $A=(a_{ij}, 1 \leq i \leq
    2k, 1 \leq j \leq n$) to be
    $$
    t^{n} \, x_1^{a_1} \cdots x_{2k}^{a_{2k}} \quad,
    $$
    where $a_i$ is the number of ones in the $i$-th row.

    Once again we can use the transfer matrix method, or the Goulden-Jackson
    method, to find the {\it weight-enumerator}
    of the set of all matrices avoiding $H$ horizontally and $V$ vertically,
    with the above weight. This is a very complicated rational function in the
    $2k+1$ variables, $t$ and $x_1, \dots, x_{2k}$. In order to count balanced
    such matrices with $2n$ columns, we have to
    extract
    the coefficient of
    $$
    t^{2n} x_1^n \cdots x_{2k}^n \quad .
    $$

    Let's call this giant, but explicitly computable, rational function $R(x_1,
    \dots, x_{2k};t)$ then
    $$
    b_k(n)={\rm Coeff}_{t^{2n} x_1^n \cdots x_{2k}^n}  R(x_1, \dots, x_{2k};t)
    $$
    $$
    =\left(\frac{1}{2\pi i}\right)^{\!2k+1} \int  \frac{R(x_1, \dots, x_{2k};t)}{(x_1 \cdots x_{2k})^n t^{2n}} \frac{dx_1 \cdots dx_{2k} dt}{x_1 \cdots x_{2k} t} \quad .
    $$

    The integrand is holonomic in the $2k+1$ continuous variables $x_1, \dots,
    x_{2k},t$ and the one discrete variable $n$, and once again, integrating
    with respect to the $2k+1$ continuous variables leaves us, by algorithmic
    proof theory \cite{Z,AZ,K}, with a holonomic discrete function in $n$.
\end{proof}

\paragraph{Comment} Theorem 1 is the special case of Theorem 3 where the sets
of forbidden patterns $H$ and $V$ are empty. Nevertheless the simple explicit
form of the integrand is useful, as we will see below.

\section*{Symbol Crunching}

The main Maple package accompanying this article is {\tt Hardin.txt} available
from

\url{https://sites.math.rutgers.edu/~zeilberg/tokhniot/Hardin.txt}.

Let us take a tour of the main features.

\begin{itemize}
    \item
        {\tt SeqB(k,N)} uses the formula in the proof of Theorem $1$ to
        crank-out the first $N$ terms of the sequence enumerating balanced
        0-1 $2k \times 2n$ matrices for $n=1$ to  $n=N$. This is useful for
        checking with the OEIS.

        For example {\tt SeqB(2,10);} gives:
        \begin{align*}
          & 6, 90, 1860, 44730, 1172556, 32496156, 936369720, 27770358330, \\
          & 842090474940, 25989269017140, \dots
        \end{align*}
        This is a very famous sequence, listed at
        \url{https://oeis.org/A002896} as the ``number of walks with $2n$ steps
        on the cubic lattice $\mathbb{Z}^3$ beginning and ending at $(0,0,0)$.

        Can you see why these two sequences are the same?

    \item {\tt SeqB(3,10);} gives the first $10$ terms of OEIS sequence A172556
        {\tt https://oeis.org/A172556}, given there with the same description
        as ours, created by Ron Hardin, who computed $49$ terms. With our Maple
        package we were able to compute $55$ terms. In fact already $49$ terms
        suffice to {\it conjecture} a linear recurrence. See the output file

        \url{https://sites.math.rutgers.edu/~zeilberg/tokhniot/oHardin2.txt}

        Later on we will see how to derive it rigorously, without guessing.

        {\tt SeqB(4,10);} gives the first $10$ terms of OEIS sequence A172555
        {\tt https://oeis.org/A172555}, also due to Hardin, who computed $33$
        terms.

        {\tt SeqB(5,10);} gives the first $10$ terms of OEIS sequence A172557
        {\tt https://oeis.org/A172557}, also due to Hardin, who computed $24$
        terms.

        While we know from Theorem $1$ that these sequences do satisfy linear
        recurrences with polynomial coefficients, we are unable at present to
        find them. We need bigger and faster computers!

    \item {\tt GF1t(A,H,V,k,t)}: inputs an alphabet {\tt A}, sets of horizontal
        and vertical forbidden patterns {\tt H} and {\tt V} respectively,
        a positive integer {\tt k}, and a variable {\tt t}. It outputs the
        rational function whose coefficient of $t^n$ is the number of $k \times
        n$ matrices avoiding the patterns of {\tt H} in rows and the patterns
        of {\tt V} in columns, whose existence is guaranteed by Theorem~$2$.

        For example to get the rational function whose coefficient of $t^n$ is
        the number of $3 \times n$ 0-1 matrices avoiding $010$ and $101$ both
        vertically and horizontally enter:

        {\tt GF1t($\{$0,1$\}$, $\{$[0,1,0],[1,0,1]$\}$,
        $\{$[0,1,0],[1,0,1]$\}$,3,t);} \quad,

        getting right away :
        $$
        -\frac{5 t^{4}-19 t^{2}-4 t -1}{t^{4}-5 t^{2}-2 t +1} \quad .
        $$

        The first few terms are
        $$
        6, 36, 102, 378, 1260, 4374, 14946, 51384, 176238, 605022, 2076288, \dots
        $$

        Surprise! These are in the OEIS \url{https://oeis.org/A060521} for a
        different reason. They are the numbers of $3 \times n$ 0-1 matrices
        avoiding, both vertically and horizontally, the patterns $000$ and
        $111$.

        And indeed this is confirmed by our Maple package. Typing:

        {\tt GF1t($\{$0,1 $\}$, $\{$[1,1,1],[0,0,0]$\}$,
        $\{$[1,1,1],[0,0,0]$\}$,3,t);}

        gives the same output. Here is an explicit bijection between these two
        sets of 0-1 $3 \times n$ matrices. Define the bijective map that maps
        the matrix entry $m_{i,j}$ to $m_{i,j} + i + j\ ({\rm mod}\ 2)$, for
        all $1\leq i\leq3$ and $1\leq j\leq n$. In other words, use as ``mask''
        a 0-1 matrix with chessboard pattern and add it to the input matrix
        (in binary arithmetic). Clearly, every occurrence of 000 or 111 (either
        vertically or horizontally) will be mapped to 101 or 010, and vice
        versa. Thus, this map transforms each $\{010,101\}$-avoiding matrix
        into a $\{000,111\}$-avoiding one, and vice versa.

    \item {\tt GF2t(H,V,k,x,t)}: inputs sets of horizontal and vertical
        forbidden patterns {\tt H} and {\tt V} respectively,
        variable names {\tt x} and {\tt t}, and outputs the rational function
        in $t$ and $x_1, \dots, x_{2k}$, whose coefficient of $t^{n} x_1^{a_1}
        \cdots x_{2k}^{a_{2k}}$ gives the number of $2k \times 2n$ column-balanced
        0-1 matrices avoiding the horizontal patterns $H$ and vertical patterns $V$
        and having $a_i$ ones in row~$i$, for all $1 \leq i \leq 2k$.

        For example if  $H=V=\{010,101\}$ (as in the Not-Alone puzzles), the
        rational function for $4 \times 2n$  matrices is given in the output
        file

        \url{https://sites.math.rutgers.edu/~zeilberg/tokhniot/oHardin5.txt}

        This is already big! But once we have it, we can Taylor expand it in
        $t$, extract the coefficient of $t^{2n}$ followed by extracting the
        coefficient of $x_1^n x_2^n x_3^n x_4^n$ to get many terms, see the
        output file

        \url{https://sites.math.rutgers.edu/~zeilberg/tokhniot/oHardin5a.txt}

        However, here we can do better, by noting that there are exactly four
        columns that are both balanced and $V$-avoiding,
        \[
          \begin{pmatrix} 0 \\ 0 \\ 1 \\ 1 \end{pmatrix},\quad
          \begin{pmatrix} 1 \\ 1 \\ 0 \\ 0 \end{pmatrix},\quad
          \begin{pmatrix} 0 \\ 1 \\ 1 \\ 0 \end{pmatrix},\quad
          \begin{pmatrix} 1 \\ 0 \\ 0 \\ 1 \end{pmatrix},
        \]
        and that they must come in pairs in order to satisfy the row-balancing
        condition, i.e., there must be the same number of columns of the first
        and second type, and the same number of columns of the third and fourth
        type. Hence we introduce weights $a, a^{-1}, b, b^{-1}$ for the four
        types of columns and arrive at a three-variable rational function
        $R(a,b;t)$, see

        \url{https://sites.math.rutgers.edu/~zeilberg/mamarim/mamarimhtml/hardinC/Jan13_2025a.txt},

        whose coefficient of $a^0b^0t^{2n}$ gives the number of $4\times2n$
        balanced Not-Alone matrices. Applying creative telescoping twice to extract
        the constant coefficient $a^0b^0$ yields a linear differential equation
        (of order~$5$ and degree~$27$) for the generating function, which can be
        converted into a recurrence (of order~$10$ and degree~$21$) for the
        sequence itself, see

        \url{https://sites.math.rutgers.edu/~zeilberg/mamarim/mamarimhtml/hardinC/Jan13_2025b.txt}

        The generating function for $6 \times 2n$ 0-1 matrices avoiding
        $010,101$ both horizontally and vertically is much bigger! See the
        output file

        \url{https://sites.math.rutgers.edu/~zeilberg/tokhniot/oHardin6.txt}

        Note that the above strategy for $4\times2n$ matrices does not apply here.

        This enabled us to find the first $30$ terms, via {\it symbolic
        computation}. See:

        \url{https://sites.math.rutgers.edu/~zeilberg/tokhniot/oHardin6a.txt}

        They start with
        $$
        8, 64, 368, 2776, 25880, 251704, 2629080, 28964248, 331032312, 3907675376, \dots
        $$
        In particular the third term, $368$ is the exact number, mentioned
        above, of solutions to a $6 \times 6$ Not-Alone puzzle. See the next
        section for $70$(!) terms using numeric computations.
\end{itemize}

\section*{Recurrences for the number of balanced $2k \times 2n$ matrices}

Using the Maple package

\url{https://sites.math.rutgers.edu/~zeilberg/tokhniot/SMAZ.txt}

that accompanies \cite{AZ}, one very quickly gets the following theorem.

\begin{theorem}
    Let $a(n)$ be the number of $4$ by $2n$ balanced matrices. Then:
    $$
    36 (2 n +3) (2 n +1) (n +1) a(n)
    -2 (2 n +3) \left(10 n^2+30 n +23\right) a(n +1)
    +(n+2)^{3}a(n +2) = 0.
    $$
\end{theorem}

But {\tt SMAZ.txt} was unable, with our computers, to find a recurrence for the
sequence enumerating $6$ by $2n$ balanced matrices. Amazingly, the second
author's Mathematica package \url{https://risc.jku.at/sw/holonomicfunctions/}
did it! We have the following fully rigorously-proved recurrence.

\begin{theorem}
  Let $a(n)$ be the number of $6$ by $2n$ balanced matrices. Then:
  \begin{align*}
    & 51200 (2n +7) (2n +5) (2n +3) (2n +1) (n+2) (n+1) \bigl(33 n^{2}+242 n +445\bigr) \, a(n) \\
    & {}- 128 (2n +7) (2n +5) (2n +3) (n+2) \bigl(7491 n^{4}+84898 n^{3}+351364 n^{2} \\
    &\qquad{}+628997 n +414370\bigr) \, a(n+1) \\
    & {}+16 (2n +5) (2n +7) \bigl(2772 n^{6}+48048 n^{5} +344379 n^{4}+1307394 n^{3} \\
    & \qquad{}+2775099 n^{2}+3125336 n +1460132\bigr) \, a(n+2) \\
    & {}+2 (2n+7) (n+3) \bigl(3201 n^{6}+61886 n^{5}+497179 n^{4}+2124170 n^{3}+5089654 n^{2} \\
    & \qquad{}+6484024 n +3431096\bigr) \, a(n+3) \\
    & {}-(n+3) (n +4)^5 \bigl(33 n^{2}+176 n +236\bigr) \, a(n+4) \, = 0 \quad.
  \end{align*}
\end{theorem}

For $k\geq4$ it seems impossible to determine the recurrence for $b_k(n)$
rigorously by creative telescoping, at least with our software and computers.
Instead, we can try to empirically find recurrences by fitting a large number
of sequence terms into a suitably chosen ansatz. For $k=4$ this approach was
successful, yielding a conjectured (but absolutely certain) recurrence:
\begin{conjecture}\label{conj:b4n}
  The number of $8$ by $2n$ balanced matrices satisfies a linear recurrence
  of order~$9$ with polynomial coefficients of degree~$36$, which is too large
  to be printed here, but which can be found on our website

  \url{https://sites.math.rutgers.edu/~zeilberg/mamarim/mamarimhtml/hardinC/b4rec.txt}.
\end{conjecture}
We do not expect that the guessing approach can deliver recurrences for $b_k(n)$
with $k\geq5$ in the near future, because already Conjecture~\ref{conj:b4n} posed
considerable challenges: Note that a naive ansatz for a recurrence of this size
contains $(9+1)\cdot(36+1)=370$ unknowns, hence $379$ terms would be required to
generate a sufficient number of linear equations. In contrast, we were only to able to
get $150$ terms, see the next section. Even the commonly-used technique of
order-versus-degree-trading---where one first guesses recurrences of non-minimal
order but much lower degree, and then constructs the minimal-order recurrence via
gcd computations---did not work here as it needed at least $266$ terms (and we estimate
that with our C program this would take 500 years and require a supercomputer with
18 TB of memory). Instead, we employed a recently-developed guessing
procedure~\cite{KK22} that is based on the celebrated LLL lattice reduction algorithm.
We found that the minimal number of terms of A172555 that are necessary
to find the order-$9$ and degree-$36$ recurrence with this guesser is~$110$. It
is interesting to note that the bit size of the guessed recurrence (after applying
an ``offset shift'' and counting only its integer coefficients) is $46{,}599$, which
comes quite close to the bit size $70{,}955$ of the first $110$ terms that were used
for guessing. Despite the fact that the recurrence stated in Conjecture~\ref{conj:b4n}
has ``ugly'' (i.e., large, up to $67$ decimal digits!) integer coefficients, we have
strong evidence that it is correct: its polynomial coefficients have quite a few
small (linear) factors, the recurrence is also valid for terms that were not used
for guessing, and continuing the sequence by unrolling the recurrence produces
only integers (at least up to $n=10000$) and not a single term with a denominator,
as one would expect for a random artifact.

\section*{Number Crunching}

Since it is unrealistic to try and find recurrences for enumerating $2k \times
2n$ balanced matrices for $k \geq 5$, it would be nice to extend, as far as our
computers would allow, Hardin's already impressive computational feats. Note
that a brute force approach is doomed.

To that purpose we have a C program available from

\url{https://sites.math.rutgers.edu/~zeilberg/mamarim/mamarimhtml/hardinC/balmat4p.c}

that extended Hardin's sequences quite a bit. The program computes, for
$n=1,2,\dots,2N$, the coefficients of the polynomial $e_k(x_1,\dots,x_{2k})^n$,
and whenever $n$ is even, outputs the coefficient of $(x_1\cdots x_{2k})^n$.
For ${\bf a}=(a_1,\dots,a_{2k})$ let
$$
  c_n({\bf a}) := {\rm Coeff}_{x_1^{a_1}\cdots x_{2k}^{a_{2k}}} \, e_k(x_1,\dots,x_{2k})^n \quad .
$$
The trivial identity
$e_k(x_1,\dots,x_{2k})^n = e_k(x_1,\dots,x_{2k}) \cdot e_k(x_1,\dots,x_{2k})^{n-1}$
immediately yields a recursive definition of these coefficients. Let
${\cal S}:=\bigl\{(s_1,\dots,s_{2k})\in\{0,1\}^{2k} \mathrel{\big|} s_1+\cdots+s_{2k}=k \bigr\}$
denote the support of $e_k(x_1,\dots,x_{2k})$, then 
$c_n({\bf a}) = \sum_{{\bf s}\in{\cal S}} c_{n-1}({\bf a}-{\bf s})$.
In this formula, one has to apply the boundary conditions $c_{n-1}({\bf a}-{\bf s})=0$
whenever ${\bf a}-{\bf s}$ has a negative component, or one that is larger than $n-1$.
Thanks to the symmetry in the variables $x_1,\dots,x_{2k}$, and thanks to the fact that
$e_k(x_1,\dots,x_{2k})^n$ is a homogeneous polynomial of degree $kn$, it suffices to
store $c_n({\bf a})$ for $n\geq a_1\geq\cdots\geq a_{2k}\geq0$ and
$a_1+\cdots+a_{2k}=kn$. Moreover, if we fix the number $N$ of desired terms from the
very beginning, we can impose the additional condition $a_i\leq N$. Since these
vectors~${\bf a}$ do not any more form a rectangular (multi-dimensional) array,
we flatten it to a one-dimensional array, in order to handle it more easily in
the C language. Conversion between these two data structures can be done by
a suitable rank and unrank function. Finally the whole computation is done
modulo prime numbers, using 64-bit integers. A sufficient number of primes can be
determined by the trivial upper bound $\binom{2k}{k}^{2n} \geq c_{2n}(n,\dots,n)$,
the latter being the $n$-th term of the sequence.

\begin{itemize}
\item If you want to see $150$ terms of the sequence enumerating $8$ by $2n$ 0-1
matrices  with row sums $4$ and column sums $n$, in other words OEIS sequence A172555 (Hardin only had 33 terms) see the output file

\url{https://sites.math.rutgers.edu/~zeilberg/mamarim/mamarimhtml/hardinC/data4.txt}.

\item If you want to see $50$ terms of the sequence enumerating $10$ by $2n$ 0-1 arrays with row sums $5$ and column sums $n$, in other words OEIS sequence A172557 (Hardin only had $24$ terms) see the file

\url{https://sites.math.rutgers.edu/~zeilberg/mamarim/mamarimhtml/hardinC/data5.txt}.

\item If you want to see $39$ terms of the sequence
enumerating $12$ by $2n$ 0-1 matrices with row sums $6$ and column sums $n$, in other words OEIS sequence A172558 (Hardin only had 19 terms) see the file:

\url{https://sites.math.rutgers.edu/~zeilberg/mamarim/mamarimhtml/hardinC/data6.txt}.

\item If you want to see $30$ terms of the sequence
enumerating $14$ by $2n$ by 0-1 matrices with row sums $7$ and column sums $n$, in other words OEIS sequence A172559 (Hardin only had 17 terms) see the file:

\url{https://sites.math.rutgers.edu/~zeilberg/mamarim/mamarimhtml/hardinC/data7.txt}.

\item If you want to see $25$ terms of the sequence
enumerating $16$ by $2n$ by 0-1 matrices with row sums $8$ and column sums $n$, in other words OEIS sequence A172560 (Hardin only had 14 terms) see the file:

\url{https://sites.math.rutgers.edu/~zeilberg/mamarim/mamarimhtml/hardinC/data8.txt}.

\item If you want to see $22$ terms of the sequence
    enumerating $18$ by $2n$ 0-1 matrices with row sums $9$ and column sums $n$, in other words OEIS sequence A172554 (Hardin only had 12 terms) see the file:

    \url{https://sites.math.rutgers.edu/~zeilberg/mamarim/mamarimhtml/hardinC/data9.txt}.

\item So far for the `vanilla case'. Above, using the Maple package {\tt Hardin.txt} we were able to find $30$ terms
of the motivating sequence of this paper, i.e., the number of balanced $6$ by $2n$ 0-1 matrices avoiding the patterns $010$ and $101$
both vertically and horizontally. Using the $C$ program mentioned above we now have $70$ terms. See the output file:

\url{https://sites.math.rutgers.edu/~zeilberg/mamarim/mamarimhtml/hardinC/dataNA3.txt}.

\item The sequence counting the number of balanced $8$ by $2n$ 0-1 matrices
  avoiding the patterns $010$ and $101$ both vertically and horizontally starts
  as follows: $18, 324, 2776, 34586, \dots$, where $34586$ gives the number of
  $8\times8$ Not-Alone puzzles. Due to the increased computational complexity
  (e.g., number of states), we attain only $16$ terms, which took 22 CPU hours
  and required almost 1 TB of memory, the latter being our limiting factor.
  See the output file

  \url{https://sites.math.rutgers.edu/~zeilberg/mamarim/mamarimhtml/hardinC/dataNA4.txt}

\end{itemize}

{\bf Conclusion}:  Humankind, and even computerkind, will most probably {\bf
never} know the exact number of $100 \times 100$ 0-1 matrices with row- and
columns- sums
all equal to $50$, but it is fun to try and see how far we can go. The OEIS
created, by our hero Neil Sloane, is an ideal platform for publishing these
hard-to-compute numbers.

\centerline{\bf Happy 85\textsuperscript{th} birthday, Neil. May you live to
see the OEIS with 1,200,000 sequences!}

\printbibliography

\bigskip
\hrule
\bigskip
Robert Dougherty-Bliss, Department of Mathematics, Dartmouth College,
Email: {\tt robert dot w dot bliss at gmail dot com}.

\bigskip

Christoph Koutschan, Johann Radon Institute for Computational and Applied
Mathematics, Austrian Academy of Sciences, Altenberger Strasse 69, 4040 Linz,
Austria. Email {\tt christoph dot koutschan at oeaw dot ac dot at}.

\bigskip

Natalya Ter-Saakov, Department of Mathematics, Rutgers University (New
Brunswick), Hill Center, Busch Campus, 110 Frelinghuysen Rd., Piscataway, NJ
08854-8019, USA. \hfill\break Email: {\tt nt399 at rutgers dot edu}.

\bigskip

Doron Zeilberger, Department of Mathematics, Rutgers University (New
Brunswick), Hill Center, Busch Campus, 110 Frelinghuysen Rd., Piscataway, NJ
08854-8019, USA. \hfill\break Email: {\tt DoronZeil at gmail dot com}.

\bigskip
Posted: Oct. 10, 2024,
\end{document}